\documentclass[letterpaper]{article}
\usepackage[margin = 1in]{geometry}
\usepackage{graphicx,color}
\usepackage[cp850]{inputenc}
\usepackage[T1]{fontenc}
\usepackage[english]{babel}
\usepackage{cite}
\usepackage{rotating}
\usepackage[f]{esvect}
\usepackage{amsmath, amsthm, amssymb}
\usepackage{rotating}
\usepackage{lscape}
\usepackage{mathrsfs}
\usepackage{url}
\def\tablenotes{\bgroup\parfillskip=0pt plus 1fil
\leftskip=0pt\relax \rightskip=0pt
\vskip2pt\footnotesize}
\def\endtablenotes{\vskip1pt\egroup}

\def\sphline{\noalign{\vskip3pt}\hline\noalign{\vskip3pt}}

\newtheorem{theorem}{Theorem}[section]

\newtheorem{lemma}[theorem]{Lemma}

\newtheorem{definition}[theorem]{Definition}

\allowdisplaybreaks

\begin{document}

\title{The Double Exponential Sinc Collocation Method for Singular Sturm-Liouville Problems}

\author{P. Gaudreau$^1$, R. Slevinsky$^2$ and H. Safouhi$^1$\footnote{Corresponding author: hsafouhi@ualberta.ca \newline The corresponding author acknowledges the financial support for this research by the Natural Sciences and Engineering Research Council of Canada~(NSERC) - Grant 250223-2011.}\\
$^1$Mathematical Section, Facult\'e Saint-Jean, University of Alberta\\
8406, 91 Street, Edmonton, Alberta T6C 4G9, Canada\\
\\
$^2$S2.29 Mathematical Institute, University of Oxford\\
Andrew Wiles Building, Radcliffe Observatory Quarter\\
Woodstock Road, Oxford UK OX2 6GG}

\date{}

\maketitle

\vspace*{0.5cm}

\vspace*{0.5cm}
{\bf \large Abstract}. \hskip 0.10cm Sturm-Liouville problems are abundant in the numerical treatment of scientific and engineering problems. In the present contribution, we present an efficient and highly accurate method for computing eigenvalues of singular Sturm-Liouville boundary value problems. The proposed method uses the double exponential formula coupled with Sinc collocation method. This method produces a symmetric positive-definite generalized eigenvalue system and has exponential convergence rate. Numerical examples are presented and comparisons with single exponential Sinc collocation method clearly illustrate the advantage of using the double exponential formula.

\vspace*{0.5cm}

\section{Introduction}~\label{Introduction}

Sturm-Liouville equations are abundant in the numerical treatment of scientific and engineering problems. For example, Sturm-Liouville equations describe the vibrational modes of various systems, such as the energy eigenfunctions of a quantum mechanical oscillator, in which case the eigenvalues correspond to the energy levels. Sturm-Liouville problems arise directly as eigenvalue problems in one space dimension. They also commonly arise from linear PDEs in several space dimensions when the equations are separable in some
coordinate system, such as cylindrical or spherical coordinates. Classical methods for computing the eigenvalues of singular Sturm-Liouville problems often rely on approximations of the differential equations using finite-difference techniques or Pr\"{u}fer transformations in order to obtain a matrix eigenvalue system \cite{Nassif1987}. Other alternatives where coefficient functions of the given problem are approximated by piecewise polynomial functions were also introduced \cite{Pruess1973}. Asymptotic methods also surfaced as an efficient tool to evaluate higher order eigenvalues \cite{Pruess1995}.

Recently, new algorithms based on collocation and spectral methods have become increasingly popular and have shown great promise \cite{Auzinger2006}. More specifically, Sinc collocation methods (SCM) \cite{Jarratt1990a, Alquran2010} have been shown to yield exponential convergence. The SCM have been used extensively during the last 30 years to solve many problems in numerical analysis. Their applications include numerical integration, linear and non-linear ordinary differential equations, partial differential equations, interpolation, and approximations to functions~\cite{Stenger-23-165-81, Stenger-121-379-00}. The SCM consists of expanding the solution of a Sturm-Liouville problem using a basis of Sinc functions. By evaluating the resulting approximation at the Sinc collocation points, one arrives at a matrix eigenvalue problem or generalized matrix eigenvalue problem for which the resulting eigenvalues are approximations to the eigenvalues of the Sturm-Liouville operator.

In~\cite{Eggert-Jarratt-Lund-69-209-87, Jarratt1990a}, a method combining the SCM and the single exponential (SE) transformation is introduced. This method, which will be referred to as SESCM leads to an efficient and accurate algorithm for computing the eigenvalues for singular Sturm-Liouville problems. The SE transformation is a conformal mapping which allows for the function being approximated by a Sinc expansion to decay single exponentially at both infinities. In~\cite{Eggert-Jarratt-Lund-69-209-87}, Eggert et al. introduced such a transformation where the resulting matrices in the generalized eigenvalue problem are symmetric and positive definite. Moreover, they were able to show that their method converges at the rate $\displaystyle {\cal O} ( N^{3/2} e^{-c\, \sqrt{N}})$.

Recently, combination of the SCM with the double exponential (DE) transformation has sparked a great interest. The DE transformation is a conformal mapping which allows for the function being approximated by a Sinc basis to decay double exponentially at both infinities. Since its introduction by Takahasi and Mori \cite{Takahasi-Mori-9-721-74}, many have studied its effectiveness in numerically evaluating integrals \cite{Sugihara2004, Mori-Sugihara-127-287-01, Sugihara2002}. As is stated in \cite{Sugihara2004}, the DE Sinc collocation method (DESCM) method yields the best available convergence for problems with end point singularities or infinite sized domains.

In the present work, we demonstrate that the DESCM leads to an extremely efficient computation of eigenvalues of singular Sturm-Liouville problems. Implementing the DESCM leads to a generalized eigenvalue problem where the matrices are symmetric and positive definite. We also show that the convergence of the DESCM algorithm is of the rate ${\cal O} \left( \frac{N^{5/2}}{\log(N)^{2}} \, e^{-\kappa \,N/\log(N)}\right)$. Our convergence result helps to explain the performance enhancement that Sinc collocation methods receive when using variable transformations with DE decay instead of SE decay. Three singular Sturm-Liouville problems are treated and comparisons with the SE transformation are presented for each example clearly illustrating the superiority of the DESCM. Lastly, we demonstrate through an example how the conformal mapping presented in the Eggert et al.'s transformation~\cite{Eggert-Jarratt-Lund-69-209-87} can be used to improve convergence of the DESCM when the coefficients functions the Sturm-Liouville problem are not analytic.

It is well known that the DE enables the Sinc expansion to achieve a much higher rate of convergence than the SE. However, it should be noted that the assumption for DE convergence is more severe than the one for SE. As in~\cite{Okayama2013, Tanaka2013}, we denote the class of functions for which SE is suitable by $\mathcal{F}_{SE}$ and the class of functions for which DE is suitable by $\mathcal{F}_{DE}$. Given the fact that $\mathcal{F}_{DE} \subsetneq \mathcal{F}_{SE}$, there exist examples such that Sinc expansion with SE achieves its usual rate, whereas it does not with DE~\cite{Okayama2013, Tanaka2013} and consequently the DESCM is not better than the SESCM for functions in $\mathcal{F}_{SE} \backslash \mathcal{F}_{DE}$. However, in~\cite{Okayama2013, Tanaka2013}, the authors present a theoretical convergence analysis for Sinc methods with DE for functions in $\mathcal{F}_{SE} \backslash \mathcal{F}_{DE}$ for which DE does not achieve its usual rate of ${\cal O} \left(  e^{- \kappa_{1} n/\log(\kappa_{2} \, n)}\right)$, and they were able to prove that DE still works for these functions with errors bounded by ${\cal O} \left(  e^{- \kappa_{3} \sqrt{N}/\log(\kappa_{4} N)}\right)$ which is slightly lower than the rate of SE; however, as stated in~\cite{Okayama2013, Tanaka2013} one can consider that there is almost no difference between the two transformations. This result also illustrates the great advantage of using DE over SE.

\section{Definitions and basic properties}
The sinc function is defined by the following expression:
\begin{equation} \label{formula: sinc functions}
\textrm{sinc}(z) = \dfrac{\sin(\pi z)}{\pi z},
\end{equation}
where $z \in \mathbb{C}$ and the value at $z=0$ is taken to be the limiting value ${\rm sinc}(0)=1$.

For $j \in \mathbb{Z}$ and $h$ a positive number, we define the Sinc function $S(j,h)(x)$ by:
\begin{equation}
S(j,h)(x) = \textrm{sinc}\left( \dfrac{x-jh}{h}\right).
\end{equation}

One of the most important properties of Sinc functions is their discrete orthogonality which is given by:
\begin{equation}
S(j,h)(kh)  =  \delta_{j,k}   \qquad \textrm{for} \qquad j,k \in \mathbb{Z},
\end{equation}
where $\delta_{j,k}$ is the Kronecker delta function.

\begin{definition}\cite{Stenger-23-165-81}
Given any function $v$ defined everywhere on the real line and any $h>0$, the Sinc expansion of $v$ is defined by the following series:
\begin{equation}\label{formula: Sinc expansion}
C(v,h)(x) = \sum_{j=-\infty}^{\infty} v_{j,h} \, S(j,h)(x),
\end{equation}
where $v_{j,h} = v(jh)$.
\end{definition}

The truncated Sinc expansion of $v$ is defined by the following series:
\begin{equation}\label{formula: y in sinc function}
C_{M,N}(v,h)(x)  = \displaystyle \sum_{j=-M}^{N} v_{j,h} \, S(j,h)(x).
\end{equation}

In~\cite{Stenger-23-165-81}, a class of functions which are successfully approximated by Sinc expansions is introduced. Now, we shall present the definition for this class of functions.

\begin{definition} \cite{Stenger-23-165-81} \label{defintion: Bd function space}
Let $d>0$ and let $\mathscr{D}_{d}$ denote the strip of width $2d$ about the real axis:
\begin{equation}
\mathscr{D}_{d} = \{ z \in \mathbb{C} : |\,\Im (z)|<d \}.
\end{equation}
In addition, for $\epsilon \in(0,1)$, let $\mathscr{D}_{d}(\epsilon)$ denote the rectangle in the complex plane:
\begin{equation}
\mathscr{D}_{d}(\epsilon) = \{z \in \mathbb{C} : |\,\Re(z)|<1/\epsilon, \, |\,\Im (z)|<d(1-\epsilon) \}.
\end{equation}
Let ${\bf B}_{2}(\mathscr{D}_{d})$ denote the family of all functions $g$ that are analytic in $\mathscr{D}_{d}$, such that:
\begin{equation}\label{formula: integral imaginary}
\displaystyle \int_{-d}^{d} | \,g(x+iy)| \, \textrm{d}y \to 0 \qquad \textrm{as} \qquad x \to \pm \infty,
\end{equation}
and such that:
\begin{equation}\label{formula: integral Npg}
 \mathcal{N}_{2}(g,\mathscr{D}_{d}) = \displaystyle \lim_{\epsilon \to 0} \left(  \int_{\partial \mathscr{D}_{d}(\epsilon)}  |\,g(z)|^{2}\, |\textrm{d}z| \right)^{1/2} <\infty.
\end{equation}
\end{definition}

The Sturm-Liouville (Sturm-Liouville) equation in Liouville form is defined as follows:
\begin{align} \label{formula: sturm-liouville problem}
Lu(x) & =  - u^{\prime \prime}(x) + q(x) u(x)  \,=\,  \lambda \rho(x)u(x) \nonumber \\
&  \hskip -0.5cm a < x < b  \qquad \qquad  u(a) = u(b)=0,
\end{align}
where $ -\infty \leq a < b \leq \infty$. Additionally, the function $q(x)$ is assumed non-negative and the weight function $\rho(x)$ is assumed positive. The values $\lambda$ are known as the eigenvalues of the Sturm-Liouville equation. The Sturm-Liouville equation is classified as either regular or singular depending on the endpoints $a$ and $b$~\cite{Amrein2005}.

In \cite{Eggert-Jarratt-Lund-69-209-87}, Eggert et al. demonstrate that when using Sinc expansion approximations for solving the Sturm-Liouville boundary value problem~\eqref{formula: sturm-liouville problem}, an appropriate change of variables results in a symmetric discretized system. The change of variable they propose is of the form~\cite{Eggert-Jarratt-Lund-69-209-87}:
\begin{equation}
v(x) = \left(\sqrt{ (\phi^{-1})^{\prime} } \, u \right) \circ \phi(x) \qquad \Longrightarrow \qquad  u(x)  =  \dfrac{ v \circ \phi^{-1}(x)}{\sqrt{ (\phi^{-1}(x))^{\prime}}},
\label{formula: EggertSub}
\end{equation}
where $\phi^{-1}(x)$ a conformal map of a simply connected domain in the complex plane with boundary points $a\neq b$ such as $\phi^{-1}(a)=-\infty$ and $\phi^{-1}(b)=\infty$. Examples of such conformal maps are given in Table~\ref{TABLE: 3}, where conformal maps inducing SE decay are given as $\phi_{SE}$ and DE decay as $\phi_{DE}$.

\begin{table}[!h]
\centering
\caption{Table of exponential variable transformations.}
\begin{tabular*}{\hsize}{@{\extracolsep{\fill}}ccc} \sphline
~~~~~~Interval &  $\phi_{SE}$ &  $\phi_{DE}$~~~ \\ \sphline
~~~~~~$(0,1)$ & $\dfrac{1}{2} \tanh(t) + \dfrac{1}{2} $ & $\dfrac{1}{2} \tanh(\sinh(t)) + \dfrac{1}{2} $~~~  \\
~~~~~~$(0,\infty)$ & $\mathrm{arcsinh}(e^{t})$ & $\mathrm{arcsinh}(e^{\sinh(t)})$~~~  \\
~~~~~~$(-\infty,\infty)$ & $t$ &  $\sinh(t)$~~~ \\ \hline
\end{tabular*}
\label{TABLE: 3}
\end{table}

Applying the change of variable \eqref{formula: EggertSub} to~\eqref{formula: sturm-liouville problem}, one obtains~\cite{Eggert-Jarratt-Lund-69-209-87}:
\begin{equation}\label{formula: transformed sturm-liouville problem}
\mathcal{L} \, v(x) =  - v^{\prime \prime}(x) + \tilde{q}(x) v(x) = \lambda \rho(\phi(x))(\phi^{\prime}(x))^{2} v(x),
\end{equation}
where:
\begin{equation}
\tilde{q}(x)  =  - \sqrt{\phi^{\prime}(x)} \, \dfrac{{\rm d}}{{\rm d} x} \left( \dfrac{1}{\phi^{\prime}(x)} \dfrac{{\rm d}}{{\rm d} x}( \sqrt{\phi^{\prime}(x)})  \right) + (\phi^{\prime}(x))^{2} q(\phi(x)).
\end{equation}

To apply the SCM method, one begins by approximating the solution of~\eqref{formula: transformed sturm-liouville problem} by the truncated Sinc expansion~\eqref{formula: y in sinc function} where the terms $v_{j,h}$ are unknown scalar weights and $h$ is a mesh size.

Inserting~\eqref{formula: y in sinc function} into~\eqref{formula: transformed sturm-liouville problem} and collocating at the Sinc points, we obtain the following system:
\begin{align}
\mathcal{L} \, C_{M,N}(v,h)(x_{k}) & = \displaystyle \sum_{j=-M}^{N} \left[ - \dfrac{{\rm d}^2}{{\rm d} x_{k}^{2}} S(j,h)(x_{k})  + \tilde{q}(x_{k}) S(j,h)(x_{k}) \right] v_{j,h} \nonumber \\
& = \mu \displaystyle \sum_{j=-M}^{N} S(j,h)(x_{k}) (\phi^{\prime}(x_{k}))^{2} \rho(\phi(x_{k})) v_{j,h},
\end{align}
where $x_{k}=kh$ for $k = -M, \ldots, N$ and $\mu$ is the approximation of the eigenvalue $\lambda$ in~\eqref{formula: transformed sturm-liouville problem}.

If we let $\delta^{(l)}_{j,k}$ be the $l^{\rm th}$ Sinc differentiation matrix with unit mesh size~\cite{Stenger1997a}:
\begin{equation}
\delta^{(l)}_{j,k} =  h^{l} \left. \left( \dfrac{\rm d}{{\rm d}x} \right)^{l} S(j,h)(x) \right|_{x=kh},
\end{equation}
then we obtain equivalently:
\begin{equation}
\sum_{j=-M}^{N} \left[ -\dfrac{1}{h^{2}} \, \delta^{(2)}_{j,k} + \tilde{q}(kh) \, \delta^{(0)}_{j,k}\right] v_{j,h}  = \mu  \displaystyle \sum_{j=-M}^{N} \delta^{(0)}_{j,k} (\phi^{\prime}(kh))^{2}  \rho(\phi(kh)) v_{j,h}.
\label{formula: index solution}
\end{equation}

Equation~\eqref{formula: index solution} can be rewritten in a matrix form as follows:
\begin{align}\label{formula: matrix solution}
\mathcal{L} \, {\bf C}_{M,N}(v,h) &  = {\bf A}{\bf v} \,=\,  \mu {\bf D}^{2}{\bf v} \quad \Longrightarrow \quad ({\bf A} - \mu {\bf D}^{2} ){\bf v} \,=\, 0,
\end{align}
where the vectors ${\bf v}$ and ${\bf C}_{M,N}(v,h)$ are given by:
\begin{align}
{\bf v} & = (v(-Nh),\ldots, v(Nh))^{T}
\nonumber\\
{\bf C}_{M,N}(v,h) & = (C_{M,N}(v,h)(-Mh), \ldots, C_{M,N}(v,h)(Nh) )^{T}.
\label{EQVECTORCN001}
\end{align}

The entries $A_{j,k}$ of the $(N+M+1) \times (N+M+1)$ matrix ${\bf A}$ are given by:
\begin{equation}\label{formula: H components}
A_{j,k} =   -\dfrac{1}{h^{2}} \, \delta^{(2)}_{j,k} + \tilde{q}(kh) \, \delta^{(0)}_{j,k} \qquad {\rm with} \qquad -M \leq j,k \leq N,
\end{equation}
and the entries $D^{2}_{j,k}$ of the $(N+M+1) \times (N+M+1)$ diagonal matrix ${\bf D}^{2}$ are given~by:
\begin{equation}\label{formula: D components}
D^{2}_{j,k} =  (\phi^{\prime}(kh))^{2} \rho(\phi(kh)) \, \delta^{(0)}_{j,k}  \qquad {\rm with} \qquad -M \leq j,k \leq N.
\end{equation}

To obtain nontrivial solutions for {\bf v}, we set:
\begin{equation}
\det({\bf A}-{\bf D}^{2}\mu)=0.
\end{equation}

To find an approximation of the eigenvalues of~\eqref{formula: transformed sturm-liouville problem}, one simply has to solve this generalized eigenvalue problem. From this it follows that there is no need to find the solution $v(x)$ of~\eqref{formula: transformed sturm-liouville problem} in order to find its eigenvalues. However, most modern eigensolvers can give eigenvalues and eigenvectors at the same time.

To implement SESCM, one needs to find a function $\phi$ for the substitution~\eqref{formula: EggertSub} that would result in the solution of~\eqref{formula: transformed sturm-liouville problem} to decay single exponentially. In~\cite{Eggert-Jarratt-Lund-69-209-87}, an upper bound for the error between the eigenvalues $\lambda$ in~\eqref{formula: transformed sturm-liouville problem} and their approximations $\mu$ in~\eqref{formula: matrix solution} is obtained when $|v(t)| \leq C \exp( -\alpha |t|))$ for some $\alpha>0$ on the real line. The upper bound is given~by~\cite{Eggert-Jarratt-Lund-69-209-87}:
\begin{equation}
|\mu-\lambda| \leq K_{v,d} \, \sqrt{\delta \lambda} N^{3/2} \exp(-\sqrt{\pi d \alpha N}),
\label{EQERRORBOUND}
\end{equation}
where $K_{v,d}$ is a constant that depends on $v$ and $d$. The optimal step size $h$ is given~by:
\begin{equation}
h = \left(\dfrac{\pi d}{\alpha N}\right)^{1/2}.
\end{equation}

In~\cite{Eggert-Jarratt-Lund-69-209-87}, Eggert et al. consider the case where $|v(t)| \leq C \exp( -\alpha |t|))$. For the more general case where $ |v(t)| \leq C \exp( -\alpha |t|^{\rho}))$ for some $\rho>0$, the optimal step size is given~by~\cite{Sugihara2002}:
\begin{equation}\label{formula: single exp mesh size}
h = \left( \dfrac{\pi d}{(\alpha N)^{\rho}} \right)^{\frac{1}{\rho+1}},
\end{equation}
and in this case equation~\eqref{EQERRORBOUND} becomes~\cite{Sugihara2002}:
\begin{equation}
|\mu-\lambda| = {\cal O} \left( \exp(-(\pi d \alpha N)^{\frac{\rho}{\rho+1}}) \right) \qquad \textrm{as}\qquad  N\to \infty.
\end{equation}

\section{The double exponential Sinc collocation method (DESCM)}

In the DE transformation, the function $v(x)$ decays double exponentially at the endpoints of its domain.

Similarly to the SESCM, we approximate the solution $v(x)$ of~\eqref{formula: transformed sturm-liouville problem} by the truncated Sinc expansion~\eqref{formula: y in sinc function}.

To analyse the convergence of the DESCM method, we need to consider the error of the second derivative of the truncated Sinc expansion of the solution $v(x)$:
\begin{equation}\label{formula: truncated sinc second derivative}
\dfrac{\textrm{d}^{2}}{\textrm{d}x^{2}}(C_{M,N}(v,h)(x)) = \sum_{k=-M}^{N} \left(  v_{j,h} \dfrac{\textrm{d}^{2}}{\textrm{d}x^{2}} (S(k,h)(x)) \right).
\end{equation}

A bound for this error is established in the following lemma. First, we let $W(x)$ be the Lambert W function, $\left \lfloor{ x }\right \rfloor$ the floor function, $\left \lceil{ x }\right \rceil$ the ceiling function, and $x^+ = \max\{x,0\}$. Let also $||\!\cdot\!||_{2}$ denote the $L^{2}$ norm for Lebesgue integrable functions:
\begin{equation}
||f(x)||_{2} = \left( \int_{\mathbb{R}} |f(x)|^{2} {\rm \, d}x \right)^{1/2}.
\end{equation}

\begin{lemma}\label{theorem: Error second derivative}
Let $E_{M,N}^{(2)}(g,h)(x)$ denote the error of approximating the second derivative of a function $g$ by the second derivative of its truncated Sinc expansion:
\begin{equation}
E_{M,N}^{(2)}(g,h)(x) =  \dfrac{{\rm \, d}^{2}}{{\rm \, d}x^{2}} \left[g(x)\right] - \dfrac{{\rm \, d}^{2}}{{\rm \, d}x^{2}} \left[C_{M,N}(g,h)(x) \right].
\end{equation}
Let:
\begin{equation}\label{formula: decay rate}
|g(x)| \leq  A\begin{cases}
 \exp( - \beta_{L} \exp(\gamma_{L} |x|)) &\quad \textrm{for} \qquad  x \in (-\infty,0] \\
 \exp( - \beta_{R} \exp(\gamma_{R} |x|)) & \quad \textrm{for} \qquad x \in [0, \infty),
\end{cases}
\end{equation}
where $A,\beta_{L},\beta_{R},\gamma_{L},\gamma_{R}>0$.

Moreover, assume that $g \in {\bf B}_{2}(\mathscr{D}_{d})$ with $d \leq \dfrac{\pi}{2\gamma}$, where $\gamma = \max \{ \gamma_{L}, \gamma_{R}\}$. If the mesh size $h$ is given~by:
\begin{equation}\label{formula: optimal h}
h = \dfrac{\log(\pi d \gamma n / \beta)}{\gamma n},
\end{equation}
where:
\begin{equation}\label{formula: n , beta choice}
\left\{
\begin{array}{l}
n = M, \; N = \left\lceil{ \dfrac{\gamma_{L}}{\gamma_{R}} M \left( 1  + \dfrac{\log\left( \beta_{L}/\beta_{R} \right)}{W( \pi d\gamma_{L} M /\beta_{L})} \right)  }\right \rceil^{+}, \beta = \beta_{L} \;\; \textrm{if} \;\; \gamma_{L} > \gamma_{R} \\[0.35cm]
n = N, \; M = \left \lfloor{ \dfrac{\gamma_{R}}{\gamma_{L}} N \left( 1  + \dfrac{\log\left( \beta_{R}/\beta_{L} \right)}{W( \pi d\gamma_{R} N /\beta_{R})} \right)  }\right \rfloor^{+}, \beta = \beta_{R} \;\; \textrm{if} \;\; \gamma_{R} > \gamma_{L} \\[0.35cm]
n = M, \; N = \left\lceil{ M \left( 1  + \dfrac{\log\left( \beta_{L}/\beta_{R} \right)}{W( \pi d\gamma_{L} M /\beta_{L})} \right)  }\right \rceil, \beta = \beta_{L} \;\; \textrm{if} \;\; \gamma_{L} = \gamma_{R} \;\; \textrm{and} \;\; \beta_{L} \geq \beta_{R} \\[0.35cm]
n = N, \; M = \left \lfloor{ N \left( 1  + \dfrac{\log\left( \beta_{R}/\beta_{L} \right)}{W( \pi d\gamma_{R} N /\beta_{R})} \right)  }\right \rfloor, \beta = \beta_{R} \;\; \textrm{if} \;\; \gamma_{L} = \gamma_{R} \;\; \textrm{and} \;\; \beta_{R} \geq \beta_{L},
\end{array}\right.
\end{equation}
then:
\begin{equation}
||E_{M,N}^{(2)}(g,h)(x)||_{2} \leq K_{g,d} \, \left(\dfrac{n}{\log(n)}\right)^{5/2} \exp \left(-  \dfrac{\pi d \gamma n}{\log(\pi d \gamma n/\beta)} \right),
\label{EQERRORSECONDDERIV}
\end{equation}
where $K_{g,d}$ is a positive constant that depends on the function $g$ and $d$.
\end{lemma}

\begin{proof}
To begin, we re-write the Sinc expansion of $g$ as follows:
\begin{equation}
E_{M,N}^{(2)}(g,h)(x) = g^{\prime \prime}(x)- \sum_{k=-\infty}^{\infty} g(kh) S(k,h)^{\prime \prime}(x) + \sum_{k = N +1}^{\infty} g(kh) S(k,h)^{\prime \prime}(x) + \sum_{k = -\infty}^{-M-1} g(kh) S(k,h)^{\prime \prime}(x).\label{eq:ProofError1}
\end{equation}

The difference of the first two terms in~\eqref{eq:ProofError1} is known as the sampling or discretization error while the sum of the last two terms corresponds to the truncation error.

Using the triangle inequality, we obtain:
\begin{eqnarray}
|| E_{M,N}^{(2)}(g,h)(x) ||_{2} & \leq & \left|\left| g^{\prime \prime}(x)- \sum_{k=-\infty}^{\infty} g(kh) S(k,h)^{\prime \prime}(x) \right|\right|_{2} + \left|\left| \sum_{k = N+1}^{\infty} g(kh) S(k,h)^{\prime \prime}(x) \right|\right|_{2}
\nonumber \\
& + & \left|\left| \sum_{k = -\infty}^{-M-1} g(kh) S(k,h)^{\prime \prime}(x) \right|\right|_{2} .
\end{eqnarray}

From~\cite{Stenger-93}, we have:
\begin{align}\label{formula: sampling error}
\displaystyle \left|\left| g^{\prime \prime}(x)- \sum_{k=-\infty}^{\infty} g(kh) S(k,h)^{\prime \prime}(x) \right|\right|_{2} & \leq \, B_{g,d} \dfrac{\exp(- \pi d /h)}{h^2},
\end{align}
where $B_{g,d}$ is a constant that depends on $g$ and $d$.

In the proof of \cite{Lundin1979}, Lundin et al. derive the following result:
\begin{eqnarray}\label{formula: truncation error Lundin}
\displaystyle \left|\left| \sum_{k = N+1 }^{\infty} g(kh) S(k,h)^{\prime \prime}(x) \right|\right|_{2}
&\leq  & \dfrac{C_{g,d}}{h^{3/2}} \displaystyle \sum_{k=N+1}^{\infty}|g(kh)|.
\end{eqnarray}
for some constant $C_{g,d}$ that depends on the function $g$ and $d$.

Utilizing this result with the bound in~\eqref{formula: decay rate}, we show that:
\begin{eqnarray}\label{formula: truncation error 1}
\displaystyle \left|\left| \sum_{k = N+1 }^{\infty} g(kh) S(k,h)^{\prime \prime}(x) \right|\right|_{2}
&\leq  & F_{g,d} \dfrac{ \exp(- \beta_{R} \exp(\gamma_{R} N h)) }{h^{5/2}},
\end{eqnarray}
where $F_{g,d}$ is a constant that depends on $g$ and $d$ and similarly, we obtain the following upper bound:
\begin{equation}\label{formula: truncation error 2}
\left|\left| \sum_{k = -\infty}^{-M-1} g(kh) S(k,h)^{\prime \prime}(x) \right|\right|_{2} \leq G_{g,d} \dfrac{ \exp(- \beta_{L} \exp(\gamma_{L} M h)) }{h^{5/2}},
\end{equation}
where $G_{g,d}$ is a constant that depends on $g$ and $d$.

Equating the exponential terms in~\eqref{formula: truncation error 1} and~\eqref{formula: truncation error 2} and solving for $N$ or $M$ in $\mathbb{N}_{0}$, we obtain:
\begin{equation}\label{formula: M collocation points with h}
\begin{cases}
N =  \left\lceil{ \dfrac{\gamma_{L}}{\gamma_{R}} M + \dfrac{\log\left( \beta_{L}/\beta_{R} \right)}{\gamma_{R} h}  }\right \rceil^{+} & \textrm{if} \quad \gamma_{L} > \gamma_{R} \\[0.35cm]
 M =\left \lfloor{ \dfrac{\gamma_{R}}{\gamma_{L}} N + \dfrac{\log\left( \beta_{R}/\beta_{L} \right)}{\gamma_{L} h}   }\right \rfloor^{+} & \textrm{if} \quad \gamma_{R} > \gamma_{L} \\[0.35cm]
N =\left\lceil{  M + \dfrac{\log\left( \beta_{L}/\beta_{R} \right)}{\gamma_{R} h}  }\right \rceil & \textrm{if} \quad \gamma_{L} = \gamma_{R} \quad \textrm{and} \quad \beta_{L} \geq \beta_{R} \\[0.35cm]
M = \left \lfloor{  N + \dfrac{\log\left( \beta_{R}/\beta_{L} \right)}{\gamma_{L} h}   }\right \rfloor  & \textrm{if} \quad \gamma_{L} = \gamma_{R} \quad \textrm{and} \quad \beta_{R} \geq \beta_{L}.
\end{cases}
\end{equation}

As can be seen from~\eqref{formula: M collocation points with h}, $M$ and $N$ depend upon the step size~$h$.

Combining~\eqref{formula: sampling error}, \eqref{formula: truncation error 1} and~\eqref{formula: truncation error 2}, we obtain:
\begin{equation}\label{formula: error bound 2}
|| E_{M,N}^{(2)}(g,h)(x) ||_{2} \leq B_{g,d} \dfrac{\exp(- \pi d /h)}{h^2} + (F_{g,d} +G_{g,d}) \dfrac{ \exp(- \beta \exp(\gamma nh)) }{h^{5/2}},
\end{equation}
where:
\begin{equation}\label{formula: collocation points M}
\begin{cases}
n = M , \quad \beta = \beta_{L} & \textrm{if} \quad \gamma_{L} > \gamma_{R} \\
n = N , \quad \beta = \beta_{R} & \textrm{if} \quad \gamma_{R} > \gamma_{L} \\
n = M , \quad \beta = \beta_{L} & \textrm{if} \quad \gamma_{L} = \gamma_{R} \quad \textrm{and} \quad \beta_{L} \geq \beta_{R} \\
n = N , \quad \beta = \beta_{R} & \textrm{if} \quad \gamma_{L} = \gamma_{R} \quad \textrm{and} \quad \beta_{R} \geq \beta_{L}.
\end{cases}
\end{equation}

Equating the exponential terms in the RHS of~\eqref{formula: error bound 2} and solving for $h$, we obtain:
\begin{equation}\label{formula: LW h optimal}
h = \dfrac{W( \pi d \gamma n /\beta)}{\gamma n}.
\end{equation}

Substituting this result in~\eqref{formula: M collocation points with h}, we obtain equation~\eqref{formula: n , beta choice}.

The first term in the asymptotic expansion of the Lambert W function as $x \to \infty$ is given~by \cite{Corless1996}:
\begin{equation}
W(x) \sim \log(x) \qquad \textrm{as} \qquad x \to \infty.
\end{equation}
Consequently the asymptotic value for the mesh size $h$ as $n \to \infty$ is given~by:
\begin{equation}\label{asymptotic stepsize}
h \sim \dfrac{\log( \pi d \gamma n /\beta)}{\gamma n} \qquad \textrm{as} \quad n \to \infty.
\end{equation}

Substituting~\eqref{asymptotic stepsize} into~\eqref{formula: error bound 2} and simplifying, we obtain equation~\eqref{EQERRORSECONDDERIV}.
\end{proof}

We shall now state a theorem establishing the convergence of the eigenvalues of a discretized Sturm-Liouville problem when the solution decays double exponentially.

\begin{theorem}\label{theorem: convergence of eigenvalues}
Let $\lambda$ and $v(x)$ be an eigenpair of the transformed differential equation \eqref{formula: transformed sturm-liouville problem}. Assume there exist positive constants $A,\beta_{L},\beta_{R},\gamma_{L},\gamma_{R}$ such that:
\begin{equation}\label{formula: xi growth condtion}
|\,v(x)| \leq  A
\left\{
\begin{array}{ccc}
 \exp( - \beta_{L} \exp(\gamma_{L} |x|) & \quad \textrm{for} \quad &  x \in (-\infty,0] \\
 \exp( - \beta_{R} \exp(\gamma_{R} |x|))& \quad \textrm{for} \quad &  x \in [0, \infty ).
\end{array}
\right.
\end{equation}

If $v \in {\bf B}_{2}(\mathscr{D}_{d})$ with $d \leq \dfrac{\pi}{2\gamma}$, where $\gamma = \max \{ \gamma_{L}, \gamma_{R}\}$.

If there is a constant $\delta>0$ such that $\tilde{q}(x)\geq \delta^{-1}$ and if the optimal mesh size $h$ is given~by:
\begin{equation}
h = \dfrac{\log(\pi d \gamma n / \beta)}{\gamma n},
\end{equation}
where $n$ and $\beta$ are given by~\eqref{formula: n , beta choice}.

Then, there is an eigenvalue $\mu$ of the generalized eigenvalue problem satisfying:
\begin{equation}
|\mu-\lambda|   \leq K_{v,d} \sqrt{\delta \lambda} \left(\dfrac{n^{5/2}}{\log(n)^{2}} \right)  \exp \left(-  \dfrac{\pi d \gamma n}{\log(\pi d \gamma n/\beta)} \right) \quad \textrm{as} \quad n \to \infty,
\end{equation}
where $K_{v,d}$ is a constant that depends on $v$ and $d$.
\end{theorem}

\begin{proof}
In general, Sturm-Liouville differential equations and their transformed counterpart \eqref{formula: transformed sturm-liouville problem} have an infinite number of eigenpairs $\{(\lambda_{i}, v_{i}(x))\}_{i\in \mathbb{N}_{0}}$. Since the choice of the eigenpair is arbitrary for the procedure of this proof, we will abstain from using indices on the eigenvalues $\lambda$ as well as on the eigenfunctions $v(x)$.

First, we assume that the arbitrary eigenpair $\lambda$ and $v(x)$ of the transformed differential equation \eqref{formula: transformed sturm-liouville problem} can be normalized as follows:
\begin{equation}\label{formula: integral normalized}
\int_{-\infty}^{\infty} v(x)^2 \rho(\phi(x)) (\phi^{\prime}(x))^2 \textrm{d}x =1.
\end{equation}

This is equivalent to normalization condition for the original system:
\begin{equation}
\int_{a}^{b} u(x)^2 \rho(x) \textrm{d}x =1.
\end{equation}

Applying~\eqref{formula: transformed sturm-liouville problem} to the collocation points $x=jh$ for $-M\leq j\leq N$ leads to:
\begin{equation}\label{formula: exact dis}
\mathcal{L} \, {\bf v} \, = \, \lambda \, \mathrm{diag}( \rho \, (\phi^{\prime})^2 ) {\bf v}\, = \, \lambda {\bf D}^2{\bf v},
\end{equation}
where ${\bf v}$ is defined in~\eqref{EQVECTORCN001}, $\lambda$ is the eigenvalue corresponding to the eigenfunction $v(x)$ and the matrix ${\bf D}$ is given~by:
\begin{equation}
{\bf D}= \mathrm{diag}( \sqrt{\rho} \, (\phi^{\prime})).
\end{equation}

Taking the difference between~\eqref{formula: exact dis} and~\eqref{formula: matrix solution}, we obtain:
\begin{equation}\label{formula: difference in approximation}
\Delta {\bf v} \, = \,   \mathcal{L} \,{\bf C}_{M,N}(v,h) - \mathcal{L} \, {\bf v} \, = \,  ({\bf A} - \lambda {\bf D}^2){\bf v},
\end{equation}
where the vector ${\bf C}_{M,N}(v,h)$ is defined in~\eqref{EQVECTORCN001}.

As stated in~\cite{Bai-Demmel-Dongarra-Ruhe-Vorst-00-Templates}, since ${\bf A}$ and ${\bf D}^2$ are symmetric positive definite matrices, there exist generalized orthogonal eigenvectors ${\bf z}_{i}$ and generalized positive real eigenvalues $\mu_{-M} \leq \mu_{-M+1} \leq \ldots \leq \mu_{N}$ such that:
\begin{equation}
{\bf Z}^{T} {\bf A} {\bf Z} \,=\, \mathrm{diag}((\mu_{-M} ,\ldots, \mu_{N})), \quad {\bf Z}^{T} {\bf D}^2 {\bf Z} \,=\,  {\bf I} \quad \textrm{and} \quad {\bf A} {\bf z}_{i} \,=\,   \mu_{i} {\bf D}^2  {\bf z}_{i}.
\label{formula: general eigen problem}
\end{equation}

The matrix ${\bf Z}$ is simply a matrix with the generalized eigenvectors ${\bf z}_{i}$ as its columns.  Equations~\eqref{formula: general eigen problem} are analogous to the spectral decomposition of one symmetric positive definite matrix, i.e. when ${\bf D}^2 = {\bf I}$. However, in this case ${\bf D}^2 \neq {\bf I}$ and we are dealing with a generalized eigenvalue problem with two symmetric positive definite matrices. It is important to note that the matrices ${\bf A}$ and ${\bf D}^{2}$ generate $N+M+1$ generalized eigenvalues. Since we are only interested in the generalized eigenvalue that approximates $\lambda$, $N+M$ of these generalized eigenvalues are not useful in this proof. The following demonstration will determine a systematic way to discard these remaining $N+M$ eigenvalues. In other words, we will demonstrate that there exists a sequence of generalized eigenvalues $\{\mu_{n}\}_{n\in\mathbb{N}}$ such that this sequence converges to the eigenvalue $\lambda$.

Since all the eigenvectors $\{{\bf z}_{i}\}_{i=-M}^{N}$ are linearly independent, there exists constants $b_{i}$ such that:
\begin{equation}\label{formula: construction of w}
{\bf v} = \displaystyle \sum_{i=-M}^{N} b_{i} \, {\bf z}_{i}.
\end{equation}

Note that the values $b_{i}$ depend on the vector ${\bf v}$ and consequently on the eigenfunction $v(x)$.

Substituting \eqref{formula: construction of w} in the RHS of \eqref{formula: difference in approximation} and using \eqref{formula: general eigen problem}, we obtain:
\begin{equation}\label{formula: error sum beta}
\Delta {\bf v} = \sum_{i=-M}^{N} b_{i} (\mu_{i} - \lambda ) {\bf D}^2{\bf z}_{i}.
\end{equation}

Multiplying both sides of \eqref{formula: error sum beta} by ${\bf z}_{j}^{T}$ and utilizing the second equation in~\eqref{formula: general eigen problem}, we obtain:
\begin{equation}\label{formula: relation bewteen eigenvalues}
{\bf z}_{j}^{T} \Delta {\bf v} = b_{j} (\mu_{j} - \lambda ) \quad \textrm{for} \quad j = -M,\ldots,N.
\end{equation}

Moreover, by multiplying both sides of \eqref{formula: construction of w} by the matrix ${\bf D}^2$, and taking the inner product of the resulting vector with ${\bf v}$ and using \eqref{formula: general eigen problem}, we obtain:
\begin{equation}\label{formula : Dw norm}
|| {\bf D} {\bf v} ||_{2}^{2} = \sum_{i=-M}^{N} b_{i}^{2} \leq (N+M+1) \, b_{p}^{2} ,
\end{equation}
where:
\begin{equation}\label{formula: beta max}
b_{p} = \max_{-M\leq i \leq N }\{ |b_{i}|\}.
\end{equation}

Note that the value of $b_{p}$ depends on the vector ${\bf v}$ and consequently on the eigenfunction $v(x)$ as can be seen from~\eqref{formula: construction of w}. Moreover, the index $p$ depends on the range $(-M,\ldots,N)$. Since $v \in {\bf B}^{2}(\mathscr{D}_{d})$ and it satisfies the decay condition~\eqref{formula: xi growth condtion}, we have the following relation when applying the trapezoidal quadrature rule to~\eqref{formula: integral normalized}:
\begin{eqnarray}
1 & = & h \displaystyle \sum_{k=-M}^{N} v(jh)^2 \rho(\phi(jh)) (\phi^{\prime}(jh))^2 +\epsilon(v,M,N)
\,=\, h || {\bf D} {\bf v} ||_{2}^{2} + \epsilon(v,M,N).
\end{eqnarray}

The assumptions on $v(x)$ guarantee that $| \epsilon(v,M,N)/h| \to 0$ as $M,N\to \infty$. From this it follows that there exist $N>0$ and $M>0$ such that $| \epsilon(v,M,N)/h| \leq 1/(2h)$, and this leads to $|| {\bf D} {\bf v} ||_{2}^{2} \geq 1/(2h)$. Combining this with~\eqref{formula : Dw norm}, we obtain a lower bound for $b_{p}$:
\begin{equation}\label{formula: beta p bound}
(2(N+M+1)h)^{-1/2} \leq b_{p} .
\end{equation}

Using the Rayleigh principle \cite{Agarwal-Bohner-Wong-99-153-99} and the assumption that there exists a constant $\delta>0$ such that $\tilde{q}(x)\geq \delta^{-1}$ implies that:
\begin{eqnarray}\label{formula: delta upper bound}
\delta^{-1} & \leq & \min_{\alpha_{j} \in \sigma({\bf A})} \{\alpha_{j}\}
\,=\,  \min_{{\bf x} \neq 0} \left \{ \dfrac{ {\bf x}^{T}{\bf A} {\bf x}}{ {\bf x}^{T}{\bf x}} \right \}
\,\leq \,  \dfrac{ {\bf z}_{j}^{T}{\bf A} {\bf z}_{j}}{ ||{\bf z}_{j}||_{2}^{2}},
\end{eqnarray}
where $\sigma({\bf A})$ is the eigenspectrum of the matrix ${\bf A}$ and $\{\alpha_{j}\}_{j=-M}^{N}$ are its eigenvalues.

The first equality in \eqref{formula: general eigen problem} indicates that $ {\bf z}_{i}^{T} {\bf A} {\bf z}_{i} = \mu_{i}$ and when combined with~\eqref{formula: delta upper bound} leads to the following upper bound for $||{\bf z}_{j}||_{2}^{2}$:
\begin{equation}\label{formula: z norm bound}
||{\bf z}_{j}||_{2}^{2} \leq \delta \mu_{j}.
\end{equation}

Let $\theta_{j}$ be the angle between ${\bf z}_{j}$ and $\Delta {\bf v}$ in \eqref{formula: relation bewteen eigenvalues}. Taking the absolute value of~ \eqref{formula: relation bewteen eigenvalues} and expanding the inner product, we obtain:
\begin{equation}\label{formula: difference eigen}
|\mu_{j}-\lambda| = \dfrac{||{\bf z}_{j}||_{2} ||\Delta {\bf v}||_{2} |\cos( \theta_{j})|}{|b_{j}|}.
\end{equation}

Since we are looking to minimize the left hand side of \eqref{formula: difference eigen}, we have to select the index $j$ such that we maximize the denominator of the right hand side of \eqref{formula: difference eigen}. Choosing the same index $l=p$ in \eqref{formula: beta max} will certainly achieve this goal given fixed $N$ and $M$. Hence, replacing $||{\bf z}_{p}||_{2}$ by its bound in \eqref{formula: z norm bound}, $|b_{p}|$ by its bound in~\eqref{formula: beta p bound} and taking into consideration that $|\cos( \theta_{p})|\leq 1$, we obtain:
\begin{equation}\label{formula: difference in eigenvalue}
|\mu_{p}-\lambda| \leq \sqrt{\delta \mu_{p}} \,  (2(N+M+1)h)^{1/2} \, ||\Delta {\bf v}||_{2}.
\end{equation}

In the following, we write $\mu$ instead of $\mu_p$ for simplicity.

We now have to consider two cases. For fixed $N$ and $M$, we have:
\begin{eqnarray}
|\mu-\lambda| \leq \lambda \quad  & \Rightarrow & \quad  \mu  =  \mu-\lambda + \lambda  \leq  |\mu-\lambda| + \lambda \leq 2 \lambda
\nonumber\\
|\mu-\lambda| > \lambda \quad  & \Rightarrow & \quad \mu  =  \mu-\lambda + \lambda  \leq  |\mu-\lambda| + \lambda \leq 2|\mu-\lambda|.
\end{eqnarray}

Combining these inequalities with \eqref{formula: difference in eigenvalue} leads to the following results:
\begin{eqnarray}\label{formula : error eigenvalues}
|\mu-\lambda| & \leq & 2 \sqrt{\delta |\mu-\lambda|} \,  ((N+M+1)h)^{1/2} \, ||\Delta {\bf v}||_{2} \quad \textrm{when} \quad  |\mu-\lambda| > \lambda
\nonumber\\
|\mu-\lambda| & \leq &  2 \sqrt{\delta \lambda} \,  ((N+M+1)h)^{1/2} \, ||\Delta {\bf v}||_{2} \qquad \quad \textrm{~~when} \quad  |\mu-\lambda| \leq \lambda.
\end{eqnarray}

Next, we will consider the quantity $||\Delta {\bf v}||_{2}$. It is easy to show that:
\begin{eqnarray}
| \Delta v(jh)| & = & | \mathcal{L} \, C_{M,N}(v,h)(jh) - \mathcal{L} \, v(jh)|
\nonumber\\
 & = & \left| \dfrac{\textrm{d}^{2}}{\textrm{d} x^{2}}C_{M,N}(v,h)(jh) - \dfrac{\textrm{d}^{2}}{\textrm{d} x^{2}}v(jh) \right|
 \nonumber\\
 & = & | E_{M,N}^{(2)}(g,h)(jh) |.
\end{eqnarray}

Hence, using lemma \ref{theorem: Error second derivative} with:
\begin{equation}
h = \dfrac{\log(\pi d \gamma n / B)}{\gamma n},
\end{equation}
we can derive the following result:
\begin{equation}\label{formula: delta w error}
||\Delta {\bf v}||_{2}  \,\leq\,  F_{v,d} \, \left(\dfrac{n}{\log(n)}\right)^{5/2} \exp
\left(-  \dfrac{\pi d \gamma n}{\log(\pi d \gamma n/ \beta)} \right),
\end{equation}
where $F_{v,d}$ is a constant that depends on $v$ and $d$.

Combining~\eqref{formula: delta w error} with \eqref{formula : error eigenvalues}, we obtain:
\begin{align}\label{formula: eigenconvergence}
|\mu-\lambda| & \leq K_{v,d} \sqrt{\delta |\mu-\lambda|} \dfrac{n^{5/2}}{\log(n)^{2}}  \exp \left(-  \dfrac{\pi d \gamma n}{\log(\pi d \gamma n/ \beta)} \right) \quad \textrm{when} \quad  |\mu-\lambda| > \lambda
\nonumber\\
|\mu-\lambda| & \leq K_{v,d} \sqrt{\delta \lambda} \dfrac{n^{5/2}}{\log(n)^{2}}   \exp \left(-  \dfrac{\pi d \gamma n}{\log(\pi d \gamma n/ \beta)} \right) \quad \quad \textrm{when} \quad  |\mu-\lambda| \leq \lambda,
\end{align}
where $K_{v,d}$ is a constant that depends on $v$ and $d$.

Simplifying, we obtain:
\begin{eqnarray} \label{formula: eigenvalue convergence equation}
|\mu-\lambda| & \leq & K^{2}_{v,d} \,  \delta  \, \dfrac{n^5}{\log(n)^{4}}  \exp \left(-  \dfrac{2\pi d \gamma n}{\log(\pi d \gamma n/\beta)} \right) \quad \quad \textrm{when} \quad  |\mu-\lambda| > \lambda
\nonumber\\
|\mu-\lambda| & \leq & K_{v,d} \sqrt{\delta \lambda} \dfrac{n^{5/2}}{\log(n)^{2}}   \exp \left(- \dfrac{\pi d \gamma n}{\log(\pi d \gamma n/\beta)} \right) \quad \textrm{when} \quad  |\mu-\lambda| \leq \lambda.
\end{eqnarray}

The bounds in \eqref{formula: eigenvalue convergence equation} demonstrate that for fixed $n$, one of the generalized eigenvalues of the matrices ${\bf A}$ and ${\bf D^{2}}$ of size $(N+M+1) \times (N+M+1)$ will approximate the eigenvalues $\lambda$. As $n$ increases, we will create a sequence of generalized eigenvalues that converge to the eigenvalue $\lambda$. Equation \eqref{formula: eigenvalue convergence equation} also indicates that $|\mu-\lambda| \to 0 $ as $n\to \infty$ for all eigenvalues $\lambda$. Moreover, as $n$ increases, the second case in~\eqref{formula: eigenvalue convergence equation} will take precedence since $|\mu-\lambda| \leq \lambda$. Hence we obtain the following asymptotic error estimate:
\begin{equation} \label{formula: eigenvalue convergence equation 2}
|\mu-\lambda|   \leq K_{v,d} \sqrt{\delta \lambda} \left( \dfrac{n^{5/2}}{\log(n)^{2}} \right)  \exp \left(- \dfrac{\pi d \gamma n}{\log(\pi d \gamma n /\beta)} \right) \qquad  \textrm{as} \qquad  n \to \infty.
\end{equation}

Since this process can be done for any arbitrary eigenpair $\{(\lambda_{i}, v_{i}(x))\}_{i\in \mathbb{N}_{0}}$, it is clear from \eqref{formula: eigenvalue convergence equation 2} that every eigenvalue $\lambda$ will satisfy the error bound for the appropriate sequence of generalized eigenvalues $\mu$.
\end{proof}

The dependence on the value of $\lambda$ in the right-hand side of \eqref{formula: eigenvalue convergence equation 2} demonstrates that convergence for eigenvalues on the lower end of the eigenvalue spectrum will be slightly faster. Nevertheless, the exponential term decreases very rapidly to 0 as $n \to \infty$ regardless the value of $\lambda$.

\section{Numerical Discussion}

In the following section, we will investigate the convergence of the DESCM compared with the SESCM for various equations. Before we proceed with the examples, we would like to address the choice of the optimal mesh for the DESCM. As shown in \cite{Trefethen2013}, the use of the mesh size in~\eqref{formula: LW h optimal} instead of~\eqref{formula: optimal h} often leads to markedly superior results for intermediate values of N. Moreover, both these formulas for the mesh size $h$ will lead to the same asymptotic error estimate in Theorem~\ref{theorem: convergence of eigenvalues}. The matrices ${\bf A}$ and ${\bf D}^2$ are constructed using~\eqref{formula: H components} and \eqref{formula: D components} respectively.
To measure the performance of the DESINC method when the generalized eigenvalues of interest are known analytically, we use the absolute error as follows:
\begin{equation}
\textrm{Absolute error} = |\mu_{i}(n)-\lambda_{i}| \qquad \textrm{for} \qquad n, i = 1,2,\ldots,
\end{equation}
where $\mu_{i}(n)$ is the $n^{\textrm{th}}$ approximation to the $i^{\textrm{th}}$ eigenvalue $\lambda_{i}$.
For the example 4.3, since the exact generalized eigenvalues are not known analytically, we computed approximations to absolute errors as follows:
\begin{equation}
\textrm{Absolute error approximation} = |\mu_{i}(n)-\mu_{i}(n-1)| \qquad \textrm{for} \qquad n, i = 1,2,\ldots,
\end{equation}
where $\mu_{i}(n)$ and $\mu_{i}(n-1)$ are the $n^{\textrm{th}}$ and $(n-1)^{\textrm{th}}$  approximations to the $i^{\textrm{th}}$ eigenvalue $\lambda_{i}$ respectively.

The codes are written in double precision using the programming language Julia \cite{Bezanson2012} and are available upon request. The eigenvalue solvers in Julia utilize the famous linear algebra package {\it LAPACK} \cite{Anderson1999}. To produce our figures, we use the Julia package {\it Winston}~\cite{Nolta2013}.

\subsection{Bessel Equation} \label{Bessel Equation section}
The Bessel equation \cite{Amrein2005} for $n \geq 1$ is defined by:
\begin{align}\label{Bessel Equation}
-u^{\prime \prime}(x) + \dfrac{4n^2-1}{x^2} \, u(x) & = \, \lambda u(x) ,\qquad 0< x <1,  \nonumber\\
u(0)  =  u(1) & = 0.
\end{align}

The solutions of \eqref{Bessel Equation} are  $u_{m}(x) = x^{1/2} J_{n}(x \sqrt{\lambda_{m}})$ and $\lambda_{m}  = j_{m,n}^{2}$ for $m = 0,1,\ldots$, where $j_{m,n}$ are the positive zeros of the Bessel function $J_{n}(x)$. In this case, the point $x=0$ is a regular singular point.
The solution $u(x)$ has the following asymptotic behavior near the endpoints:
\begin{equation}
u(x) \sim \begin{cases} \displaystyle a_{1} x^{n+1/2} & \textrm{as} \quad x\to 0 \\
\displaystyle a_{2}(x-1) & \textrm{as} \quad x\to 1,
\end{cases}
\end{equation}
for some constants $a_{1}$ and $a_{2}$. 
To implement the DE transformation, we use the first mapping in Table~\ref{TABLE: 3}:
\begin{equation}
x = \phi_{DE}(t)\,=\,  \dfrac{1}{2} \tanh(\sinh(t)) + \dfrac{1}{2}
\,\sim\, \begin{cases} \displaystyle \dfrac{1}{2} \exp(-\exp(-t)) & \textrm{as} \quad t\to -\infty \\
\displaystyle 1- \dfrac{1}{2}\exp(-\exp(t)) & \textrm{as} \quad t \to  \infty.
\end{cases}
\end{equation}

Hence, the transformed equation \eqref{formula: transformed sturm-liouville problem} is given~by:
\begin{eqnarray}\label{transformed BesselDE}
-v^{\prime \prime}(t) & + & \left( \cosh^2(t)+\dfrac{1}{4} -\dfrac{3}{4}\mathrm{sech}^2(t)+ \dfrac{(4n^2-1)\cosh^2(t)}{(e^{2\sinh(t)}+1)^2}\right) \, v(t)
\nonumber\\ & = & \lambda \left( \dfrac{\cosh(t)}{2\cosh^{2}(\sinh(t))} \right)^2 v(t).
\end{eqnarray}

The solution of \eqref{transformed BesselDE} has the following asymptotic behavior near infinities:
\begin{equation}
v(t) \sim \begin{cases} \displaystyle \alpha_{1} \exp\left(\frac{t}{2} -n\exp(-t)\right) & \textrm{as} \quad t\to -\infty \\[0.25cm]
\displaystyle \alpha_{2} \exp\left(-\frac{t}{2}-\frac{1}{2} \exp(t)\right) & \textrm{as} \quad t\to \infty,
\end{cases}
\end{equation}
for some constants $\alpha_{1}$ and $\alpha_{2}$. Consequently, we can establish the following bound for $v(t)$:
\begin{equation}
| v(t)| \leq A \exp(-n \exp(|\,t|)) \qquad \textrm{for} \qquad t\in\mathbb{R},
\end{equation}
for some constant $A$.
Using~\eqref{formula: LW h optimal} with $\gamma =1$, $\beta =n$ and $d = \dfrac{\pi}{2}$, we obtain:
\begin{equation}\label{step size: Bessel}
h = \dfrac{W(\pi^2 N /2n)}{N}.
\end{equation}

Before we conclude this numerical example, we mention that nonsymmetric Sinc expansions ($M \not= N$ in the Sinc expansion) can provide numerical efficiency in problems where the solutions to the transformed Sturm-Liouville equation \eqref{formula: transformed sturm-liouville problem} have different asymptotic behaviour at both infinities. To illustrate this claim, we will compare the symmetric and nonsymmetric Sinc expansions for this example.

For the transformed Bessel equation \eqref{transformed BesselDE}, using~\eqref{formula: n , beta choice} with $B_{L} = n$, $B_{R} = 1/2$, $\gamma_{L} =1$ and $ \gamma_{R} = 1$, we obtain the following equation for the number of right collocation points:
\begin{equation}\label{formula: collocation points M Bessel}
N =  \left \lceil{ M \left( 1  + \dfrac{\log\left(2n \right)}{W( \pi^2 M/2n )} \right)  }\right \rceil.
\end{equation}

Using~\eqref{formula: LW h optimal}, we obtain:
\begin{equation}
h = \dfrac{W(\pi^2 M/2n)}{M}.
\end{equation}

Figure \ref{figure: Bessel} displays the absolute error for the symmetric and nonsymmetric DESCM and SESCM for the first eigenvalue of~\eqref{Bessel Equation} with $n=7$ and $\lambda_{1} \approx 122.9076002036162$.

It is clear from Figure~\ref{figure: Bessel} that the symmetric DESCM outperforms the SESCM and more importantly the nonsymmetric DESCM proves to be far superior compared to both methods.
\begin{figure}[!htp]
\begin{center}
\includegraphics[width=0.4\textwidth]{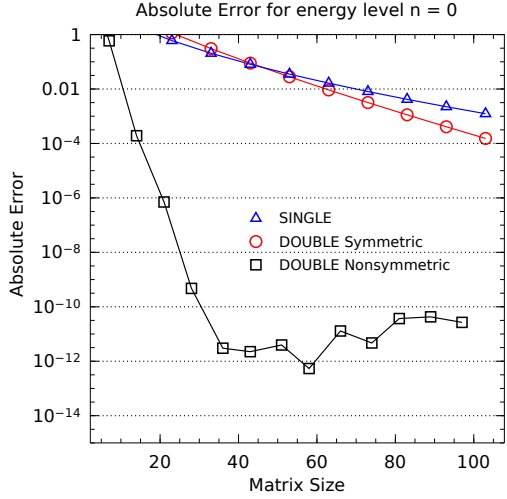}
\caption{Plot of the absolute convergence of the SESCM as well as the symmetric and nonsymmetric DESCMs for the first eigenvalue $\lambda \approx 122.9076002036162$ of~\eqref{Bessel Equation} with $n=7$.}
\label{figure: Bessel}
\end{center}
\end{figure}

\subsection{Laguerre Equation}\label{Laguerre Equation section}
The Laguerre equation in Liouville form \cite{Amrein2005} for $\alpha \in (-\infty,\infty)$ is defined by:
\begin{align}\label{Laguerre Equation}
- u^{\prime \prime}(x) + \left(\dfrac{\alpha^2-1/4}{x^2} -\dfrac{\alpha+1}{2} + \dfrac{x^2}{16} \right)\, u(x) & = \lambda u(x), \qquad 0< x <\infty
\nonumber\\ & \hskip -4cm  u(0)  =  u(\infty)  = 0.
\end{align}

The eigenvalues of~\eqref{Laguerre Equation} are $\lambda_{n} = n-1$ for $n = 1,2,\ldots$ which are independent of $\alpha$.

We will consider the case $ \alpha > \dfrac{1}{2}$ where the point $x=0$ is a regular singular point. The solution $u(x)$ has the following behavior near the endpoints:
\begin{equation}
u(x) \sim \begin{cases} A x^{\alpha+1/2} & \textrm{as} \quad x\to 0 \\
Bx^{2\lambda+\alpha+1/2}\exp \left(-\dfrac{x^2}{8} \right) & \textrm{as} \quad x\to \infty,
\end{cases}
\end{equation}
for some constants $A$ and $B$.
To implement the DE transformation, we use the second mapping in Table~\ref{TABLE: 3}:
\begin{eqnarray}
x & = & \phi_{DE}(t)
\,=\,  \mathrm{arcsinh}(e^{\sinh(t)})
\,\sim \,
\begin{cases} \exp\left[-\dfrac{\exp(-t)}{2}\right] & \textrm{as} \quad t \to -\infty \\[0.3cm]
\dfrac{\exp(t)}{2} & \textrm{as} \quad t \to  \infty.
\end{cases}
\end{eqnarray}

Hence, the transformed equation \eqref{formula: transformed sturm-liouville problem} is given~by:
\begin{eqnarray}\label{transformed LaguerreDE}
-v^{\prime \prime}(t) & + & \left[ -\dfrac{3\cosh^2(x)}{16}\left(\tanh(\sinh(x))+\dfrac{1}{3}\right)^2 +\dfrac{\cosh^2(x)}{3} +\dfrac{1}{4} - \dfrac{3}{4}\mathrm{sech}^{2}(x) \right. \nonumber \\
& + & \left. \left(\dfrac{\alpha^2-1/4}{\mathrm{arcsinh}^2(e^{\sinh(t)})} -\dfrac{\alpha+1}{2} + \dfrac{\mathrm{arcsinh}^2(e^{\sinh(t)})}{16}\right)\left(\dfrac{\cosh^{2}(t)}{1+e^{-2\sinh(t)}}\right) \right] \, v(t)\nonumber \\
& = & \left(\dfrac{\lambda\cosh^{2}(t)}{1+e^{-2\sinh(t)}}\right) v(t).
\end{eqnarray}

The solution of \eqref{transformed LaguerreDE} has the following asymptotic behavior near infinities:
\begin{equation}
v(t) \sim
\begin{cases} A^{\prime} \exp(\frac{t}{2} -\frac{\alpha}{2}\exp(-t)) & \textrm{as} \quad t\to -\infty \\
B^{\prime} \exp(t(\alpha+2\lambda)-\frac{1}{32} \exp(2t)) & \textrm{as} \quad t\to \infty,
\end{cases}
\end{equation}
for some constants $A^{\prime}$ and $B^{\prime}$. Consequently, we can establish the following bound for $v(t)$:
\begin{equation}
| v(t)| \leq \tilde{A} \exp \left(- \frac{1}{32} \exp( 2 |\,t|) \right) \qquad \textrm{for} \qquad t\in\mathbb{R},
\end{equation}
for some constant $\tilde{A}$.
Using~\eqref{formula: LW h optimal} with $\gamma =2$, $\beta = \displaystyle \frac{1}{32} $ and $d = \displaystyle \frac{\pi}{4}$, we obtain:
\begin{equation}\label{step size: Laguerre}
h = \dfrac{W(16 \pi^2 N )}{2N}.
\end{equation}

Since the solution to the transformed Laguerre equation \eqref{transformed LaguerreDE} has different asymptotic behaviour at both infinities, we can use a nonsymmetric Sinc expansion. Using~\eqref{formula: LW h optimal} with $B_{L} = \alpha/2$, $B_{R} = 1/32$, $\gamma_{L} =1$ and $ \gamma_{R} = 2$, we obtain the following equation for the number of left collocation points:
\begin{equation}\label{formula: collocation points M Laguerre}
M = \max \left\{\left \lfloor{ 2N \left( 1  - \dfrac{\log\left(16\alpha \right)}{W( 16\pi^2 N)} \right)  }\right \rfloor , 0 \right\}.
\end{equation}

The step size in this case is given~by~\eqref{formula: LW h optimal} as:
\begin{equation}
h = \dfrac{W(16\pi^2 N)}{2N}.
\end{equation}

Figure \ref{figure: Laguerre} displays the absolute error for the DESCM and SESCM for the first eigenvalue of~\eqref{Laguerre Equation} with $\alpha = 3$ and $\lambda_{1} = 0$. Here again, the nonsymmetric case performs better than the symmetric case.
\begin{figure}[!htp]
\begin{center}
\includegraphics[width=0.4\textwidth]{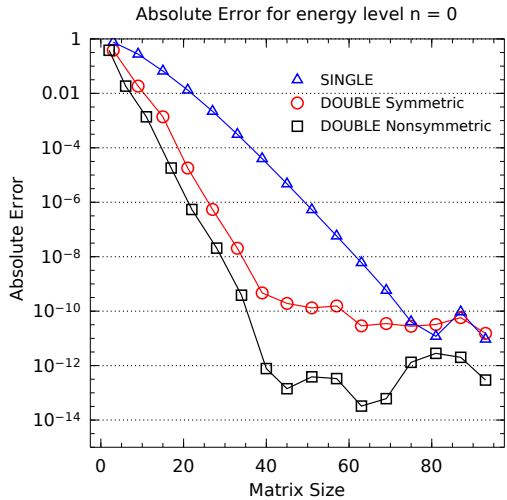}
\caption{Plot of the absolute convergence of the SESCM as well as the symmetric and nonsymmetric DESCMs for the first eigenvalue of~\eqref{Laguerre Equation} with $\alpha = 3$ and $\lambda_{1} = 0$.}
\label{figure: Laguerre}
\end{center}
\end{figure}

\subsection{Complex Singular equation}
The following example illustrates the case where the coefficients $q(x)$ and $\rho(x)$ might have complex singularities close to the real line. In such instances the DESCM still outperforms the SESCM.

The singular equation that we consider is defined by the following:
\begin{align}\label{formula: Singular potential}
-u^{\prime \prime}(x) + \left(x^2 + \dfrac{\tanh(x)}{\log(x^2 +1.1 )}\right)\, u(x) & = \dfrac{\lambda}{x^2 + \cos(x)} u(x) ,\qquad -\infty< x <\infty
\nonumber \\ u(-\infty)  =  u(\infty) & = 0.
\end{align}

Equation \eqref{formula: Singular potential} has several points where the coefficient functions are not analytic. Firstly, the coefficient function:
\begin{equation}
q(z) = z^2 + \dfrac{\tanh(z)}{\log(z^2 +1.1 )},
\end{equation}
has complex singularities at the points:
\begin{equation} \label{complex singularity 1}
z \,=\,  \pm i\sqrt{0.1}  \qquad \textrm{and} \qquad z \,=\, i\pi \left( n+\frac{1}{2}\right) \quad \textrm{for} \quad n\in \mathbb{Z}.
\end{equation}

Secondly, the weight function:
\begin{equation}
\rho(z) =\dfrac{1}{z^2 + \cos(z)},
\end{equation}
has complex singularities at the points:
\begin{equation}\label{complex singularity 2}
z \,\approx\,  \pm  1.621347946\,i  \qquad \textrm{and}  \qquad z \,\approx\,  \pm  2.593916090\,i.
\end{equation}

The solution $u(x)$ has the following behavior near the boundary points:
\begin{equation}
u(x) \sim A |x|^{-1/2} \exp \left( -\dfrac{1}{2} x^2 \right) \qquad \textrm{as} \qquad |x| \to \infty,
\end{equation}
for some constant $A$.
Since this example is not treated in literature, we will present the implementation of the SE transformation. Since the solution already exhibits SE decay, we use the third mapping in Table~\ref{TABLE: 3} $x  \,=\,   \phi_{SE}(t)  \,=\,   t$ to implement the SE transformation. Consequently, the transformed equation \eqref{formula: transformed sturm-liouville problem} is exactly the same as \eqref{formula: Singular potential}. Moreover, we can obtain a bound for the solution of \eqref{formula: transformed sturm-liouville problem}, which is given~by:
\begin{equation}
| v(t)| \leq \tilde{A}  \exp \left( -\dfrac{1}{2} t^2 \right) \quad \textrm{for}  \quad t \in \mathbb{R}.
\end{equation}

Due to the complex singularities in~\eqref{complex singularity 1} and \eqref{complex singularity 2}, the optimal value for the strip width is $d = \sqrt{0.1}$. Hence using~\eqref{formula: single exp mesh size} with $\rho = 2$ and $\beta = \dfrac{1}{2}$, we obtain:
\begin{equation}
h = \left(\dfrac{4\pi\sqrt{0.1}}{N^2}\right)^{1/3}.
\end{equation}

To implement the DE transformation, we use the third mapping in Table~\ref{TABLE: 3}:
\begin{eqnarray}
x & = & \phi_{DE}(t) \,=\,  \sinh(t)\,\sim\,
\begin{cases}
-\dfrac{\exp(-t)}{2}& \textrm{as} \quad t \to -\infty \\[0.15cm]
\dfrac{\exp(t)}{2} & \textrm{as} \quad t \to  \infty.
\end{cases}
\end{eqnarray}

Hence, the transformed equation \eqref{formula: transformed sturm-liouville problem} is given~by:
\begin{equation}\label{transformed FunkyDE}
-v^{\prime \prime}(t)+ \left[ \dfrac{1}{4} -\dfrac{3}{4} \mathrm{sech}^2(t) + \sinh(t)^2 + \dfrac{\tanh(\sinh(t))\cosh^{2}(t)}{\ln(\sinh^2(t) +1.1 )} \right] v(t) = \left[\dfrac{\lambda \cosh^{2}(t)}{\sinh^{2}(t)+\cos(\sinh(t))}\right] v(t).
\end{equation}

The solution of \eqref{transformed FunkyDE} has the following asymptotic behavior near infinities:
\begin{equation}
v(t) \sim A^{\prime} \exp \left( -|t| -\dfrac{1}{8} \exp(2|t|) \right) \qquad \textrm{as} \qquad |t| \to \infty,
\end{equation}
for some constants $A^{\prime}$. Consequently, $v(t)$ can be bounded as follows:
\begin{equation}
| \, v(t)| \leq \tilde{A}  \exp \left( -\dfrac{1}{8} \exp(2|t|) \right) \qquad \textrm{for} \qquad t \in \mathbb{R}.
\end{equation}

The conformal map $\phi(t) = \sinh(t)$ moves the singularities in~\eqref{complex singularity 1} and \eqref{complex singularity 2} as follows. First, the coefficient function:
\begin{equation}
\tilde{q}(z) = \dfrac{1}{4} -\dfrac{3}{4} \mathrm{sech}^2(z) + \sinh(z)^2 + \dfrac{\tanh(\sinh(z))\cosh^{2}(z)}{\ln(\sinh^2(z) +1.1 )},
\end{equation}
has complex singularities at the points:
\begin{equation} \label{complex singularity 1 moved}
\pm i \arcsin \left( \sqrt{0.1} \right), \, i \left( \dfrac{\pi}{2} + n\pi \right) \;\; \textrm{and} \;\; \pm \left( \mathrm{arccosh}\left(\dfrac{\pi}{2} + \pi n\right)+ \dfrac{\pi}{2} i \right), \;\; n\in \mathbb{Z}.
\end{equation}

Second, the weight function:
\begin{equation}
\rho(\sinh(z))\cosh^{2}(z) =\dfrac{\cosh^{2}(z)}{\sinh(z)^2 + \cos(\sinh(z))},
\end{equation}
has complex singularities at the points:
\begin{equation}\label{complex singularity 2 moved}
z \approx  \pm \left[ 1.063876028 + \dfrac{\pi}{2} i \right] \;\;\textrm{and}\;\;
z \approx  \pm  \left[ 1.606899463+ \dfrac{\pi}{2} i \right].
\end{equation}

Due to the complex singularities in~\eqref{complex singularity 1 moved} and~\eqref{complex singularity 2 moved}, the optimal value for the strip width is $d = \arcsin \left( \sqrt{0.1} \right)$.

Using~\eqref{formula: LW h optimal} with $\gamma =2$, $\beta = \displaystyle \frac{1}{8} $ and $d = \arcsin(\sqrt{0.1})$, we obtain:
\begin{equation}
h = \dfrac{W(16 \pi \arcsin(\sqrt{0.1}) N )}{2N}.
\end{equation}

As can be seen from the above analysis, the conformal map $\phi(t) = \sinh(t)$ requires the solution of~\eqref{transformed FunkyDE} to belong to ${\bf B}_{2}(\mathscr{D}_{\arcsin(\sqrt{0.1})})$. However, we will demonstrate that by choosing a conformal map of the form $\phi(t) = \kappa \sinh(t)$ for some parameter $0<\kappa<1$, we ware able to create a solution to~\eqref{transformed FunkyDE} that belongs to the function space ${\bf B}_{2}(\mathscr{D}_{\frac{\pi}{4}})$. Since $\arcsin(\sqrt{0.1}) <\frac{\pi}{4}$, by Theorem~\ref{theorem: convergence of eigenvalues} we expect eigenvalues of functions belonging to ${\bf B}_{2}(\mathscr{D}_{\frac{\pi}{4}})$ to converge faster. For more information on the use of conformal maps to accelerate convergence of Sinc numerical methods, we refer the interested reader to~\cite{Slevinsk-Olver-2-A676-15}.

To implement the double exponential transformation for~\eqref{transformed FunkyDE}, we use the mapping:
\begin{equation}
x =  \phi_{DE}(t) =  \kappa \sinh(t)\,\sim\,
\begin{cases}
-\dfrac{\kappa \exp(-t)}{2} & \textrm{as} \quad t \to -\infty \\[0.25cm]
\dfrac{\kappa \exp(t)}{2} & \textrm{as} \quad t \to  \infty
\end{cases} \;\; \textrm{with} \;\;  0<\kappa<1.
\end{equation}

Hence, the transformed equation \eqref{formula: transformed sturm-liouville problem} is given~by:
\begin{equation}\label{transformed FunkyDE kappa}
-v^{\prime \prime}(t)+ \left[ \dfrac{1}{4} -\dfrac{3}{4} \mathrm{sech}^2(t) + \kappa^2 \sinh(t)^2 + \dfrac{ \tanh(\kappa \sinh(t)) \kappa^2  \cosh^{2}(t)}{\ln( \kappa^2 \sinh^2(t) +1.1 )} \right] v(t) = \left[\dfrac{\lambda \kappa^2 \cosh^{2}(t)}{\kappa^2 \sinh^{2}(t)+\cos(\kappa \sinh(t))}\right] v(t).
\end{equation}

The solution of \eqref{transformed FunkyDE} has the following asymptotic behavior near infinities:
\begin{equation}\label{asymptotic behaviour: funkyDE kappa}
v(t) \sim A^{\prime} \exp \left( -|t| -\dfrac{\kappa^2}{8} \exp(2|t|) \right) \qquad \textrm{as} \qquad |t| \to \infty,
\end{equation}
for some constants $A^{\prime}$. Consequently, $v(t)$ can be bounded as follows:
\begin{equation}
| \, v(t)| \leq \tilde{A}  \exp \left( -\dfrac{\kappa^2}{8} \exp(2|t|) \right) \qquad \textrm{for} \qquad t \in \mathbb{R}.
\end{equation}

The conformal map $\phi(t) = \kappa\sinh(t)$ moves the singularities in~\eqref{complex singularity 1} and \eqref{complex singularity 2} as follows. Firstly, the coefficient function:
\begin{equation}
\tilde{q}(z) = \dfrac{1}{4} -\dfrac{3}{4} \mathrm{sech}^2(t) + \kappa^2 \sinh(t)^2 + \dfrac{ \tanh(\kappa \sinh(t)) \kappa^2  \cosh^{2}(t)}{\ln( \kappa^2 \sinh^2(t) +1.1 )},
\end{equation}
has complex singularities at the points:
\begin{equation} \label{complex singularity 1 moved kappa}
\pm i \arcsin \left( \dfrac{\sqrt{0.1}}{\kappa} \right),\; i \left( \dfrac{\pi}{2} + n\pi \right) \;\;\textrm{and} \;\; \pm \left( \mathrm{arccosh}\left(\dfrac{\pi}{2\kappa} + \dfrac{\pi n}{\kappa}\right)+ \dfrac{\pi}{2} i \right), \;\; n\in \mathbb{Z}.
\end{equation}

Secondly, the weight function:
\begin{equation}
\rho(\kappa \sinh(z)) \kappa^2 \cosh^{2}(z) =\dfrac{\kappa^2 \cosh^{2}(z)}{ \kappa^2 \sinh(z)^2 + \cos(\kappa \sinh(z))},
\end{equation}
has complex singularities at the points:
\begin{equation}\label{complex singularity 2 moved kappa}
z \approx  \pm \left[\mathrm{arccosh}\left(\dfrac{1.621347946}{\kappa} \right) + \dfrac{\pi}{2} i \right] \;\; \textrm{and} \;\;
z \approx  \pm  \left[ \mathrm{arccosh}\left(\dfrac{2.593916090}{\kappa} \right) + \dfrac{\pi}{2} i \right].
\end{equation}

By Theorem~\ref{theorem: convergence of eigenvalues}, the optimal value for the strip width $d$ can be at most $\dfrac{\pi}{4}$. Hence, by choosing $\kappa = \sqrt{0.2}$, the closest singularities of~\eqref{transformed FunkyDE kappa} lie on the lines $y=\pm i\dfrac{\pi}{4}$. Consequently, using~\eqref{formula: LW h optimal} with $\gamma =2$, $\beta = \displaystyle \frac{0.2}{8} $ and $d = \dfrac{\pi}{4}$, we obtain:
\begin{equation}
h = \dfrac{W(20 \pi^2 N)}{2N}.
\end{equation}

Figure \ref{figure: Funky} displays the convergence rate of the DESCM and SESCM in computing approximations of the the first eigenvalue $\lambda \approx 0.690894228848$ of the singular equation \eqref{formula: Singular potential}. It is clear that the convergence is further improved by using the adapted transformation $\phi(t) = \sqrt{0.2}\sinh(t)$.
\begin{figure}[!htp]
\begin{center}
\includegraphics[width=0.4\textwidth]{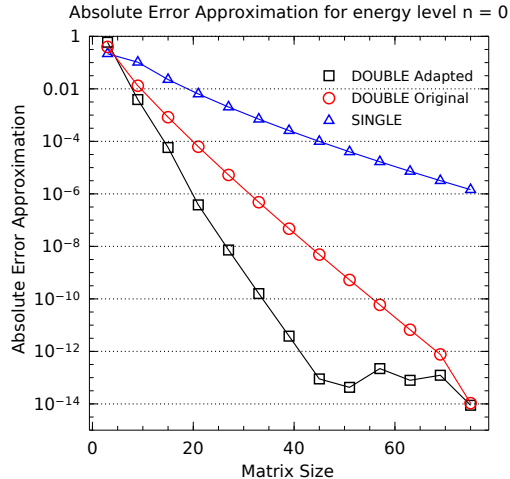}
\caption{Plot of the absolute convergence of the SESCM as well as the symmetric and adapted DESCMs for the first eigenvalue $\lambda \approx 0.690894228848$ of~\eqref{formula: Singular potential}.}
\label{figure: Funky}
\end{center}
\end{figure}

\section{Conclusion}

Computing the eigenvalues of singular Sturm-Liouville equations can be numerically challenging. In this work, we compute the eigenvalues of such equations using the Sinc-collocation method coupled with double exponential variable transformation. The implementation of the DESCM leads to a generalized eigenvalue problem with symmetric and positive definite matrices. In addition, we also show that the convergence of the DESCM is of the rate ${\cal O} \left( \frac{N^{5/2}}{\log(N)^{2}}  e^{-\kappa N/\log(N)}\right)$ for some $\kappa>0$, as $N\to \infty$ where $2N+1$ is the dimension of the resulting generalized eigenvalue system. Consequently, DESCM outperforms SESCM proposed in~\cite{Eggert-Jarratt-Lund-69-209-87}. We follow up this claim by conducting numerical studies of several Sturm-Liouville eigenvalue problems using both the SESCM and the DESCM. Finally, we also use adapted conformal mappings to accelerate the convergence of the DESCM to ensure the analyticity of the transformed coefficient functions in a strip of maximal width. In all our numerical examples, we were able to reach an unprecedented degree of accuracy.

\bibliography{C:/AMyBibliography/My_Bibliography,C:/AMyBibliography/library}

\end{document}